\documentclass{amsart}
\usepackage[margin = 1.5in]{geometry}

\usepackage[english]{babel}
\usepackage[utf8]{inputenc}
\usepackage{hyperref}
\usepackage{amsfonts}
\usepackage{cite}
\usepackage{enumerate}
\usepackage{amssymb}
\usepackage{graphicx} 
\usepackage{commath}
\usepackage{physics}
\usepackage{cases}
\usepackage[dvipsnames]{xcolor} 
\usepackage{scrextend}
\usepackage{bm}
\usepackage{nicematrix}
\usepackage{setspace}
\usepackage{afterpage}
\usepackage[labelfont=bf]{caption}
\usepackage[normalem]{ulem}

\usepackage{amsthm}
\theoremstyle{plain}
\newtheorem{theorem}{Theorem}
\newtheorem*{theorem*}{Theorem}

\newtheorem{proposition}{Proposition}
\newtheorem{lemma}{Lemma}
\newtheorem{corollary}[]{Corollary}
\theoremstyle{definition}

\newtheorem{claim}{Claim}

\newtheoremstyle{mythm}%
{3pt}
{3pt}
{}
{}
{\bfseries}
{.}
{.5em}
{}%
\theoremstyle{mythm}
\newtheorem{remark}{Remark}
\newtheorem*{remark*}{Remark}

\newcommand\xqed[1]{%
  \leavevmode\unskip\penalty9999 \hbox{}\nobreak\hfill
  \quad\hbox{#1}}
\newcommand\demo{\xqed{$\blacktriangleleft$}}

\newcommand{\e}{\varepsilon}

\newcommand{\B}{\mathbb{B}}

\newcommand{\F}{\mathbb{F}}

\newcommand{\N}{\mathbb{N}}

\newcommand{\R}{\mathbb{R}}
\newcommand{\bS}{\mathbb{S}}

\newcommand{\Z}{\mathbb{Z}}

\newcommand{\calL}{\mathcal{L}}

\newcommand{\calO}{\mathcal{O}}

\newcommand{\calQ}{\mathcal{Q}}

\newcommand{\cups}{\cup \cdots \cup}
\newcommand{\wedges}{\wedge \cdots \wedge}
\newcommand{\timess}{\times\cdots\times}

\DeclareMathOperator{\Supp}{supp}

\DeclareMathOperator{\lspan}{span}

\DeclareMathOperator{\argmin}{arg\, min}

\DeclareMathOperator{\Gr}{\bm{Gr}}

\renewcommand{\(}{\left(}
\renewcommand{\)}{\right)}

\makeatletter
\newcommand*{\rom}[1]{\expandafter\@slowromancap\romannumeral #1@}
\makeatother

\date{10 January 2022}
\title{Multijoints And Factorisation}
\author{Michael Chi Yung Tang}
\address{School of Mathematics and Maxwell Institute for Mathematical Sciences,
University of Edinburgh, Peter Guthrie Tait Road, Kings Buildings, Edinburgh EH9 3FD}

\email{michael.tang@ed.ac.uk}

\begin{document}
\begin{abstract}
We solve the dual multijoint problem and prove the existence of so-called ``factorisations'' for arbitrary fields and multijoints of $k_j$-planes. 
 More generally, we deduce a discrete analogue of a theorem due in essence to Bourgain and Guth. Our result is a universal statement which describes a property of the discrete wedge product without any explicit reference to multijoints and is stated as follows: Suppose that $k_1 + \ldots + k_d = n$. There is a constant $C=C(n)$ so that for any field $\F$ and for any finitely supported function $S : \mathbb{F}^n \rightarrow \mathbb{R}_{\geq 0}$, there are factorising functions $s_{k_j} : \mathbb{F}^n\times \Gr(k_j, \F^n)\rightarrow \mathbb{R}_{\geq 0}$ such that 
 \[(V_1 \wedges V_d)S(p)^d \leq C\prod_{j=1}^d s_{k_j}(p, V_j),\]
 for every $p\in \F^n$ and every tuple of planes $V_j\in \Gr(k_j, \F^n)$, and 
 \[\sum_{p\in \pi_j} s(p, e(\pi_j)) =\norm{S}_d\]
 for every $k_j$-plane $\pi_j\subset \F^n$, where $e(\pi_j)\in \Gr(k_j,\F^n)$ denotes the translate of $\pi_j$ that contains the origin and $\wedge$ denotes the discrete wedge product. 
\end{abstract}

\maketitle
\section{Introduction}
The multijoint problem is the discrete analogue of the multilinear Kakeya problem. This discrete problem is cast in affine space over an arbitrary field $\F$, finite or otherwise. Consider the projective space of lines through the origin in $\F^d$,\footnote{The degree of multilinearity $d$ and ambient dimension $n$ coincide and so we denote them both by $d$ until Subsection \ref{section:results1}.} which we denote by $\bS^{d-1}$, and its elements, which we call \textbf{directions}. If $l\subset \F^d$ is a line, let $e(l) \in \bS^{d-1}$, the direction of $l$, be the translate of $l$ which contains the origin. Let $ \omega_j\in \bS^{d-1}$ for $1\leq j \leq d$. We define the \textbf{discrete wedge product} $\wedge_{j=1}^d \omega_j = \omega_1 \wedges \omega_d = 1$ if the directions $\omega_1, \ldots, \omega_d$ span $\F^d$, and $0$ otherwise. Then, we can define the \textbf{multijoint kernel} $\delta$ by
\[\delta(p,l_1, \ldots, l_d) :=  \(\prod_{j=1}^d \chi_{l_j}(p)\) e(l_1)\wedges e(l_d),\]
for lines $l_j \subset \F^d$ and points $p\in \F^d$. For each $1\leq j \leq d$ let $\calL_j$ be a set of lines in $\F^d$ and we define set of their \textbf{multijoints} by
\[J := \left\{ p\in \F^d : \exists (l_j)_j\in\calL_1\timess\calL_d \text{ s.t. } \delta(p,l_1, \ldots, l_d)=1\right\}.\]
Then the \textbf{multijoint problem} consists of establishing the inequality
\begin{equation}\label{eq:06}\sum_{p\in \F^d} \(\sum_{l_1\in \calL_1}\cdots \sum_{l_d\in\calL_d} \delta(p,l_1, \ldots, l_d)f_1(l_1) \cdots f_d(l_d)\)^{\frac{1}{d-1}} \lesssim \prod_{j=1}^d \(\sum_{l_j\in\calL_j}f_j(l_j)\)^{\frac{1}{d-1}}\end{equation}
for arbitrary $f_j : \calL_j \rightarrow \R_{\geq0}$, with implicit constants independent of $\calL_j$ and $f_j$. This was first proved by Zhang, \cite{Zha16}.
\subsection{Multijoint Inequalities as Boundedness of Operators}\label{section:operators}
We interpret the multijoint inequality as boundedness of a multilinear operator and give a proof of this boundedness by establishing two assertions.

We define the  \textbf{multijoint operator} by
\[T[f_1, \ldots, f_d](p)= \sum_{l_1\in \calL_1}\cdots \sum_{l_d\in\calL_d} \delta(p,l_1, \ldots, l_d)f_1(l_1)\cdots f_d(l_d),\]
for $p \in \F^d$ and arbitrary $f_j : \calL_j\rightarrow \R_{\geq0}$. Then we can express the multijoint inequality \eqref{eq:06} as the boundedness of the multilinear operator $T$, 
\begin{equation}\label{eq:multiOp}\norm{T[f_1, \ldots, f_d]^{\frac 1 d}}_{L^{\frac{d}{d-1}}(J)} \leq C \prod_{j=1}^d \norm{f_j}_{L^1(\calL_j)}^{\frac 1 d},\end{equation}
for arbitrary $f_j : \calL_j \rightarrow \R_{\geq0}$, where we use $\norm{f}_{L^p(A)}$ to denote $\(\sum_{a\in A}\abs{f(a)}^p\)^\frac 1 p$ for any discrete set $A$.

Now, consider an arbitrary non-negative test function $S : J \rightarrow \R_{\geq0}$. Suppose that we can find a ``factorising'' function $s : J\times \bS^{d-1} \rightarrow \R_{\geq0}$ so that 
\[T[f_1, \ldots, f_d](p)S(p)^d  \leq C^d \sum_{l_1\in\calL_1}\cdots\sum_{l_d\in\calL_d} \prod_{j=1}^d s(p, e(l_j))f_j(l_j)\]
uniformly over all $f_j :\calL_j \rightarrow \R_{\geq0}$ - this is the first assertion mentioned above. Defining positive linear operators $T_j; L^1(\calL_j) \rightarrow L^1(J)$ by 
\[T_j[f_j](p) =  \sum_{l_j\in\calL_j}s(p,e(l_j))f_j(l_j),\]
 for any $p\in J$, let us further suppose that $s$ can be chosen so that the operators $T_j$ each satisfy 
\[\norm{T_j[f_j]}_{L^1(J)}\leq \norm{S}_d\norm{f_j}_{L^1(\calL_j)}\]
 for all $f_j : \calL_j\rightarrow \R_{\geq0}$ - this is the second assertion. Then 
\[\sum_{p\in J}S(p)T[f_1, \ldots, f_d](p)^{\frac 1 d}\leq C \sum_{p\in J} \prod_{j=1}^d T_j [f_j](p)^{\frac1 d}\leq C\prod_{j=1}^d \norm{T_j[f_j]}_1^\frac 1 d \leq C\norm{S}_d\prod_{j=1}^d \norm{f_j}_1^\frac1d,\]
where we have used H\"older's inequality followed by boundedness of each $T_j$. Therefore, to prove \eqref{eq:multiOp}, and hence \eqref{eq:06}, it suffices, for arbitrary $S$, to find bounded linear operators $T_j$, as above, so that 
\[S(p)T[f_1, \ldots, f_d](p)^{1/d} \leq C T_1[f_1](p)^{1/d}\cdots T_d[f_d](p)^{1/d}\]
for all $f_j : \calL_j \rightarrow \R_{\geq0}$ and $p\in J$ for some constant $C=C(d)$.

This type of analysis is closely related to so-called geometric multilinear duality, \cite{CHV20}. Such methods were used by Carbery and Valdimarsson in their proof of Guth's endpoint mulitlinear Kakeya theorem, \cite{CarVal12}. Importantly, to analyse the multijoint problem in terms of functional operators, it is crucial we bound $ST^{1/d} $ by the geometric mean of the operators $T_j$, which contains precisely $d$ factors and weights that depend on $d$, only. That we bound $T$ from above by a geometric mean is what motivates the description of the results described in this article as \textbf{factorisation} theorems.

This analysis was motivated by the following theorem, essentially due to Bourgain and Guth, which formed the cornerstone of their proof of the general multilinear Kakeya theorem. The precise formulation of this result was not stated by Bourgain and Guth, but was given in \cite{CarVal12} by Carbery and Valdimarsson. 

To state these factorisation results we momentarily use $\bS^{d-1}$ to denote the classical unit sphere in $\R^d$. For any non-zero vectors $\omega_1, \ldots, \omega_d\in \bS^{d-1}$ we define the \textbf{Euclidean wedge product} $\omega_1 \wedges \omega_d$ to be the unsigned volume of the parallelepiped with edges $\omega_1, \ldots, \omega_d$.

\begin{theorem*}[Multilinear Kakeya Factorisation Theorem, \cite{Gut08, BouGut10, CarVal12}]\label{theorem:BG}
Let $\calQ$ be the lattice of unit cubes in $\R^d$ and let $S: \calQ \rightarrow [1, \infty)$ be finitely supported. Then there exists a function $s : \calQ \times \bS^{d-1} \rightarrow \R_{\geq 0}$ so that
\[(\omega_1 \wedges\omega_d) S(Q)^d\lesssim \prod_{j=1}^ds(Q, \omega_j)\]  
for all $Q\in\calQ$ and $\omega_j\in\bS^{d-1}$, where $\wedge$ denotes the Euclidean wedge product, and so that
\[\sum_{Q : T\cap Q \neq\emptyset } s(Q,e(T)) \leq \norm{S}_d\] 
for any tube $T\subset \R^d$ with unit cross sectional area and direction $e(T)$.
\end{theorem*}

The function $s$ that appears in the multilinear Kakeya factorisation theorem is constructed in terms of the so-called visibility and directional surface area of a suitable polynomial hypersurface that can associated to the given configuration of tubes.

\subsection{Results}\label{section:results1}

The multijoint problem is a discrete analogue of the multilinear Kakeya problem. This can be seen from \eqref{eq:06} by taking $\F=\R$, replacing the counting measure with the Lebesgue measure, the discrete wedge product (in the definition of $\delta$) with the absolute value of the Euclidean wedge product, and lines with 1-tubes. It was observed in \cite{CarVal12} (and implicity in \cite{Gut08,BouGut10}) that the multilinear Kakeya theorem follows from the multilinear Kakeya factorisation theorem, above. Although the multijoint problem was proved by Zhang, \cite{Zha16}, and more recently, higher-dimensional generaliations were proved by Tidor, Yu and Zhao, \cite{TYZ20}, a discrete analogue to the multilinear Kakeya factorisation theorem has remained unproven, until now. 

From now on we will fix our notation so that $n$ denotes the underlying spacial dimension and $d$ denotes the degree of multilinearity.

Suppose $V_1, \ldots, V_d$ are $k_1$-, $\ldots, k_d$-dimensional vector spaces in $\F^n$, respectively, where $k_1 + \ldots + k_d = n$. We define the \textbf{discrete wedge product} on these $k_j$-planes by
\[\wedge_{j=1}^dV_j := \wedge_{j=1}^d\wedge_{k=1}^{k_j} \omega_{j, k},\]
where $\omega_{j, 1}, \ldots, \omega_{j, k_j} \in \bS^{n-1}$ is a choice of $k_j$ linearly independent directions contained in $V_j$, for each $1\leq j \leq d$. We can now define the (\textbf{$k_j$-plane}) \textbf{multijoint kernel} by
\[\delta(p,\pi_1, \ldots, \pi_d) = \(\prod_{j=1}^d \chi_{\pi_j}(p)\)e(\pi_1)\wedges e(\pi_d),\]
for all $p\in \F^n$ and all $k_j$-planes $\pi_j$. Given sets of $k_j$-planes, $\Pi_j$, we say that $p$ is a $k_j$-\textbf{multijoint}, or multijoint in short, if there are planes $\pi_j\in \Pi_j$ so that $\delta(p,\pi_1, \ldots, \pi_d)=1$, and we say that the planes $\pi_j$ \textbf{form a multijoint at} $p$.

Let $1 \leq k \leq n$. Recall that the \textbf{Grassmannian with respect to $k$ and $\F^n$} is the set of all $k$-subspaces of $\F^n$, which we denote by $\Gr(k, \F^n)$. Given any affine $k$-plane $\pi$, let $e(\pi)\in \Gr(k, \F^n)$ denote the translate of $\pi$ that contains the origin.

We may now state our results as follows.
\begin{theorem}[Discrete Factorisation Theorem]\label{corollary:01}
Let $k_1 + \ldots + k_d = n$. For all finitely supported $S:\F^n\rightarrow \R_{\geq0}$ with $\norm{S}_{d}=1$, for each $1\leq j \leq d$ there exists a function $s_{k_j}: \F^n\times \Gr(k_j, \F^n) \rightarrow \R_{\geq0}$ for each $1\leq j \leq d$ so that
\[(V_1\wedges V_d )S(p)^d \lesssim_{n}\prod_{j=1}^ds_{k_j}(p,V_j),\]
 for all $p\in \F^n$ and $V_j\in\Gr(k_j, \F^n)$ for $1\leq j \leq d$, and so that for any $1\leq j \leq d$,
\[\sum_{p\in \pi_j}s_{k_j}(p,e(\pi_j)) \leq 1\] 
for all affine $k_j$-subspaces $\pi_j\subset \F^n$.
\end{theorem}

Theorem \ref{corollary:01} is precisely an analogue of the Multilinear Kakeya Factorisation Theorem. Restricting our attention to a particular choice of sets of $k_j$-planes, we arrive at a multijoint-specific factorisation theorem.
\begin{theorem}[Multijoint Factorisation Theorem]\label{theorem:2b}
Let $\Pi_1, \ldots, \Pi_d$ be sets of $k_1$-$, \ldots, k_d$-planes in $\F^n$, respectively, so that $k_1 + \ldots + k_d =n$. Let $J = \{p : \exists (\pi_j)_j \text{ so that } \delta(p,\pi_1, \ldots, \pi_d)=1\}$. For all finitely supported $S:J\rightarrow \R_{\geq0}$ with $\norm{S}_{d}=1$ there exists  a function $s_{k_j}: J\times \Gr(k_j, \F^n)\rightarrow \R_{\geq0}$   for each $1\leq j \leq d$ so that
\begin{equation}\label{thmeq:1}\delta(p,\pi_1, \ldots, \pi_d)S(p)^d \lesssim_{n}\prod_{j=1}^ds_{k_j}(p,e(\pi_j)),\end{equation}
 for all $p\in J$ and $(\pi_1, \ldots, \pi_d)\in \Pi_1\timess \Pi_d$, and so that
\begin{equation}\label{thmeq:2}\sum_{p\in \pi_j\cap J}s_{k_j}(p, e(\pi_j)) = 1\end{equation}
for all $\pi_j\in \Pi_j$ and all $1\leq j \leq d$. 
\end{theorem}

Note that in the case that the sets $\Pi_1, \ldots, \Pi_d$ are finite, the set of their multijoints is also finite and hence the finite support hypothesis on $S$ is redundant.

\begin{remark}\label{remark:suf}
Recall the linear operators $T_j$ from Section \ref{section:operators}. With $s$ given by Theorem \ref{theorem:2b},  inequality \eqref{thmeq:1} implies that $ST^{1/d} \lesssim T_1^{1/d}\cdots T_d^{1/d}$ and equation \eqref{thmeq:2} implies that $\norm{T_j}=1$ for each $1\leq j \leq d$. Hence the multijoint inequality, \eqref{eq:06}, follows from Theorem \ref{theorem:2b}.\demo
\end{remark}

In closely related work in collaboration with Carbery, \cite{CarTan}, we show in the case where $k_1, \ldots, k_d = 1$ (multijoints of lines) that Theorem \ref{theorem:2b} can be deduced from the assumption that \eqref{eq:06} holds. Together with Remark \ref{remark:suf}, this shows that the multijoint inequality, \eqref{eq:06}, and the multijoint factorisation theorem, Theorem \ref{theorem:2b}, are equivalent, at least in this special case.

\subsection{Overview of the Article}
In Section \ref{section:a2} we recall the notion of ``handicaps'', as introduced by Yu and Zhao in \cite{YuZha19}, and further developed in \cite{TYZ20}. Handicaps allow us to choose polynomial vanishing conditions so that we may apply the polynomial method while respecting the geometry of multijoints. 

In Section \ref{section:BG} we diverge from the approach in \cite{YuZha19, TYZ20} and adapt the recent novel developments introduced therein to establish our factorisation theorems.

\section*{Acknowledgements}
This research was supported by The Maxwell Institute Graduate School in Analysis and its Applications (MIGSAA), a Centre for Doctoral Training funded by the UK Engineering and Physical Sciences Research Council (grant EP/L016508/01), the Scottish Funding Council, Heriot-Watt University and The University of Edinburgh. I would like to thank the whole team at MIGSAA for creating such an excellent programme and fostering an inclusive and productive environment for mathematical research.

I am extremely grateful to my supervisor, Anthony Carbery, for his constant support throughout my time at MIGSAA and in preparing this article.

I would also like to thank James Wright for his help in formulating Theorem \ref{corollary:01} succinctly.

The results contained in this article are also detailed in the author's PhD thesis, \cite{Tang}.
\section*{Notation}
We write $A\lesssim B$ to mean that there is a non-negative constant $C$, depending only on dimension, so that $A\leq CB$. We write $A\lesssim_{n} B$ to mean that there is a constant $C$, depending only on $n$ and the dimension, so that $A\leq CB$. We write $B\gtrsim A$ and $B\gtrsim_{n} A$ to mean $A\lesssim B$ and $A \lesssim_{n}B$, respectively. Moreover, by $A\sim B$, we mean $A\lesssim B$ and $B \lesssim A$, and finally, by $A\sim_{n} B$, we mean $A\lesssim_{n}B$ and $B\lesssim_{n} A$.
\section{Polynomials, Handicaps and Vanishing Conditions}\label{section:a2}
\subsection{The Polynomial Method}
 This section introduces the main tools we will employ, and motivates the arguments that follow.

We will work in an arbitrary field, $\F$, and any derivative will be the \textbf{Hasse} derivative. All arguments remain valid when $\F=\R$ with the usual derivative operator. 

Let $f\in \F[x_1, \ldots, x_n]$ and $\alpha\in \N^n$ be a multi-index.  The $\alpha$-th Hasse derivative of $f$ is the coefficient of the monomial $z^\alpha = z_1^{\alpha_1}\cdots z_n^{\alpha_n}$ in the polynomial $p(x+z)\in \(\F[x_1, \ldots, x_n]\)[z_1, \ldots,z_n]$. This is denoted by $D^\alpha f$. For further details see \cite{Hasse} or \cite{CarIli20}.

Recall that for any field $\F$ and $\lambda \in \N$, 
\begin{equation}\label{eq:dim}
\dim  \F_\lambda [x_1, \ldots, x_n] = {\lambda + n \choose n},\end{equation}
 and hence, $\dim \(\(\F_\lambda [x_1, \ldots, x_n]\)^*\) = {\lambda + n \choose n}$. From this fact, we can deduce a commonly used consequence, known as the parameter counting lemma. Indeed, checking whether a (Hasse) derivative of a low degree polynomial at a point is zero is equivalent to checking whether $\langle \phi, f\rangle=0$ for an appropriately chosen $\phi \in \(\F_\lambda [x_1, \ldots, x_n]\)^*$. Therefore, the following so-called parameter counting Lemma follows from elementary linear algebra.
\begin{lemma}[Parameter Counting]
Let $\phi_1, \ldots, \phi_m$ be homogeneous linear functionals which act on $\F_\lambda[x_1, \ldots, x_n]$. If $m < {n + \lambda \choose n}$, then there is a non-zero $f\in \F_\lambda[x_1, \ldots, x_n]$ so that $\langle \phi_i, f\rangle=0$, for each $1\leq i \leq m$.
\end{lemma}

We organise our arguments in such a way that the only dimensional constants that appear in our stated results arise directly from the implied constants in the inequalities described by ${n + \lambda \choose n} \sim \lambda^n$, or, for any positive $k_1, \ldots, k_d $,
\[{k_1 + \lambda\choose k_1} \cdots {k_d + \lambda \choose k_d} \sim_n \lambda^{k_1 + \ldots + k_d}.\]

\subsection{Handicaps}\label{section:2.1}
For any finite subset $J\subseteq \F^n$, a \textbf{handicap} is a function $\alpha : J \rightarrow \Z$. We will use $\alpha$ to equip $J\times \Z_{\geq 0}$ with a linear order so that it has a least element. We will then accumulate vanishing conditions as we increment along this linear order, starting from the least element.
\begin{remark}
The handicap keeps track of degrees of freedom in choosing vanishing conditions. More precisely, given a fixed $p_0\in J$, the degrees of freedom are the $(\abs{J}-1)$ entries of $(\alpha_{p_0}-\alpha_p)_{p\in J\setminus \{p\}}$. \demo
\end{remark}
Since $J$ is finite, we may equip it with a total order so that for $p, q\in J$, at most one of $p < q$, $p=q$ or $p>q$ holds. 
We define a total order (called the \textbf{priority order} in \cite{TYZ20}), denoted by $\prec$, on $J\times \Z_{\geq 0}$ as follows. We say $(p,r) \prec (p^\prime, r^\prime)$ if
\begin{itemize}
\item $r-\alpha_p < r^\prime - \alpha_{p^\prime}$, or
\item $r-\alpha_p = r^\prime - \alpha_{p^\prime}$ and $p<p^\prime$.
\end{itemize}
Moreover, $(p,r)=(p^\prime, r^\prime)$ if and only if $p=p^\prime$ and $r=r^\prime$. We write $\preceq$ to allow for this equality case.

\subsection{Vanishing Conditions}\label{section:6.2}
Let $\Pi_1, \ldots, \Pi_d$ be finite sets of $k_1$-$, \ldots, k_d$-dimensional planes in $\F^n$, respectively, where $k_1 + \ldots + k_d= n$. Let $J = \{p : \exists (\pi_j)_j \text{ so that }\delta(p, \pi_1 \ldots, \pi_d)=1\}$ be the associated set of multijoints. Let $\lambda \in \N$ be a parameter and let $\alpha : J \rightarrow \Z$ be a handicap. Choose an ordering of the set  $J$ and thereafter, define the total order $\prec = \prec_\alpha$ on $J\times\Z_{\geq0}$, as in Section \ref{section:2.1}.

Let $\pi\subset \F^n$ be a $k$-plane which we fix for the remainder of Section \ref{section:6.2}. Without loss of generality, we assume that $\pi$ is spanned by the coordinate vectors $e_1, \ldots, e_k$, so we make the identification
\begin{equation}\label{eq:pi}\left\{f\rvert_\pi : f\in \F[x_1, \ldots, x_n]\right\} = \F[x_1, \ldots, x_k].\end{equation}
For each $(p,r)$, we define $\B_r(p,\pi,\lambda)$ to be the vector space of linear functionals which act on $\F_\lambda[x_1, \ldots, x_k]$ of the form $f\mapsto Df(p)$ for some differential operator $D$, of order $r$, acting on $k$-variate polynomials.

With $\alpha$, $\lambda$ and $\pi$ fixed, we now define sets $B(p,\pi, \alpha, \lambda)$ for $p\in \pi\cap J$. To do so, we perform an iterative procedure starting with the $\prec$-least element of $(\pi\cap J)\times \Z_{\geq0}$ and proceeding to the $\prec$-next element on each iteration. Starting with $(p,0)$, the least element of $(\pi\cap J)\times \Z_{\geq0}$, we choose a set $B_0(p,\pi,\alpha,\lambda)\subset \B_0(p,\pi,\lambda)$ that is a basis for $\B_0(p,\pi,\lambda)$. 

Assume for each $(p^\prime,r^\prime)\prec(p,r)$, we have chosen sets $B_{r^\prime}(p^\prime,\pi,\alpha,\lambda)$ so that the disjoint union, $\cup_{(p^\prime, r^\prime)\prec (p,r)} B_{r^\prime}(p^\prime, \pi, \alpha, \lambda)$ is a basis for $\lspan_{(p^\prime, r^\prime)\prec (p,r)} \B_{r^\prime}(p^\prime, \pi, \lambda)$. Thereafter, choose $B_r(p,\pi,\alpha,\lambda) \subset \B_r(p,\pi,\lambda)$ so that the disjoint union 
\[\bigcup_{(p^\prime, r^\prime)\preceq (p,r)} B_{r^\prime}(p^\prime, \pi, \alpha, \lambda)\]
 is a basis for $\lspan_{(p^\prime, r^\prime)\preceq (p,r)} \B_{r^\prime}(p^\prime, \pi, \lambda)$. 
\begin{remark}\label{remark:03}
Writing these unions out with more complete notation,
\[\bigcup_{(p^\prime, r^\prime)\preceq (p,r)} B_{r^\prime}(p^\prime, \pi, \alpha, \lambda) = \bigcup_{\substack{(p^\prime, r^\prime) \in (\pi\cap J) \times \Z_{\geq0}:\\(p^\prime, r^\prime)\preceq (p,r)}} B_{r^\prime}(p^\prime, \pi, \alpha, \lambda).\]
Observe that we can write these unions as
\[ \bigcup_{\substack{p^\prime \in (p + e(\pi))\cap J, \, r^\prime \in \Z_{\geq0}:\\(p^\prime, r^\prime)\preceq (p,r)}} B_{r^\prime}(p^\prime, \pi, \alpha, \lambda),\]
since $\pi = p+e(\pi)$ for any $k$-plane $\pi$. Since the first argument of $B$ is $p$, the only information this construction requires from the second argument of $B$ is which vector space is parallel to $\pi$, namely, $e(\pi)\in \Gr(k,\F^n)$. It follows that for this $k$-plane construction, we may also reduce the second argument to $e(\pi)$, although for notational convenience, we again continue to simply write $B(p,\pi,\alpha,\lambda)$.

Further observe that for any other $k$-plane $\pi^\prime$ so that $\pi\cap J = \pi^\prime \cap J$, we have that $B(p,\pi,\alpha,\lambda) = B(p,\pi^\prime, \alpha,\lambda)$ for all $p\in \pi\cap J$.
\demo
\end{remark}

Let
\[B(p,\pi,\alpha,\lambda) = \bigcup_{r\geq 0} B_r(p,\pi,\alpha,\lambda).\]
By construction, the disjoint union
\[\bigcup_{p\in \pi\cap J} B(p,\pi,\alpha,\lambda)\]
is a basis for the space of linear functionals on $\F_\lambda[x_1, \ldots, x_k]$. Hence, by \eqref{eq:dim},
\begin{equation}\label{eq:020}\sum_{p\in \pi\cap J}\abs{B(p,\pi,\alpha,\lambda)} = \dim \F_{\lambda}[x_1, \ldots, x_k] = {\lambda+k \choose k}.\end{equation}

Equation \eqref{eq:020} replaces the use of the Fundamental Theorem of Algebra to bound the number of zeros of a polynomial, which we often see in polynomial method arguments such as in \cite{Dvi08, Qui09}. 

Let 
\[\tilde S_k(p,\pi, \alpha,\lambda) = \abs{B(p,\pi, \alpha, \lambda)}\]
 for  $\pi \in \cup_{j : k_j = k}\Pi_j$ and $p\in \pi\cap J$. Our analysis will only consider pairs $(p,\pi)$ so that $p\in \pi$. Therefore, it is not necessary to define $\tilde S_k$ for pairs such that $p\not\in \pi$. These quantities remain central to our analysis. Moreover, it follows  immediately from \eqref{eq:020} that
 \begin{equation}\label{eq:UpperBound}\sum_{p\in \pi\cap J}\tilde S(p,\pi,\alpha,\lambda) = \dim \F_{\lambda}[x_1, \ldots, x_k] = {\lambda+k \choose k}.\end{equation}
Given a polynomial $f\in \F_\lambda[x_1,\ldots x_k]$, each set $B(p,\pi,\alpha,\lambda)$ indexes the \textbf{vanishing conditions}, $\langle \phi, f\rangle =0$ for $\phi \in \B(p,\pi,\alpha,\lambda)$.  A vanishing condition is a condition of the form $\langle \phi, f\rangle =0$ where $\phi$ is a linear form on $\F[x_1, \ldots, x_k]$. Such conditions include evaluation at a point, or evaluation of derivatives at a point.

We further note that the sets $B$ are translation-invariant with respect to $\alpha$.

\begin{proposition}[Translation Invariance]\label{proposition:3}
Let $\pi$ be a $k$-plane. Let $\alpha$ be a handicap and $c\in \Z$. Let $(\alpha + c)_p := \alpha_p + c$ for all $p\in J$. Then
\[B(p,\pi,\alpha,\lambda) = B(p,\pi,\alpha + c, \lambda)\]
for all $p\in (l\cap J)$. Furthermore, $\tilde S(p,\pi,\alpha,\lambda) = \tilde S(p,\pi,\alpha+c, \lambda)$.
\end{proposition}

\subsection{The Vanishing Lemma}
The implied constant $C = C(k_1, \ldots, k_d)$ in our Theorem \ref{theorem:2b} is derived from the polynomial method.
A traditional application of the polynomial method would find a non-zero polynomial $f\in \F[x_1, \ldots, x_d]$ of low degree which vanishes at every point of $J$. More generally, we may ask that $f$ satisfies vanishing conditions, $\langle \phi, f\rangle =0$ for some $\phi \in (\F[x_1, \ldots, x_d])^*$.  

It was observed in \cite{TYZ20} that choosing vanishing conditions in the traditional way is naïve, and unwittingly imposes more vanishing conditions than necessary. Following \cite{TYZ20}, we will begin by choosing vanishing conditions, with some degrees of freedom, in an optimal way such that a vanishing lemma remains valid. Once the vanishing lemma is established, the use of polynomials is concluded. We then show that these vanishing conditions were chosen in a way that satisfies desirable uniform boundedness, monotonicity and continuity properties with respect to these degrees of freedom. This allows for a heuristically simple perturbation argument and we conclude the argument by choosing vanishing conditions in a way that respect the geometry of the particular multijoint configuration with which we are working. In Section \ref{section:GoodHandicap}, our results diverge from \cite{TYZ20}. We make a different choice of handicap, and with it we ultimately we establish the discrete Bourgain--Guth theorem.

The sets $B(p,\pi,\alpha,\lambda)$, as constructed in Section \ref{section:6.2} above, are comprised of dual maps of the form $Df(p)$ for some derivative operators which act on polynomials restricted to $\pi$. Let us abuse notation and write $D\in B(p,\pi,\alpha,\lambda)$ to mean a derivative operator $D$, so that $D f(p) = \langle \phi, f\rangle$, for some $\phi \in B(p,\pi,\alpha,\lambda)$. Then, it is a cornerstone of the Tidor--Yu--Zhao handicap construction that for any fixed $\alpha$ and $\lambda$, the vanishing conditions indexed by $B_j(p,\pi_j,\alpha,\lambda)$, for $p\in J$ and $\pi_j\in\Pi_j$, satisfy a vanishing lemma.

\begin{lemma}[Vanishing Lemma, {\cite[Lemma 5.8]{TYZ20}}]\label{lemma:13}
Fix $\alpha$ and $\lambda$. For each $1\leq j \leq d$, $\pi_j\in \Pi_j$ and $p\in J$, build the sets $B(p,\pi_j) = B(p,\pi_j, \alpha,\lambda)$ and $B_r(p,\pi_j) = B_r(p,\pi_j, \alpha,\lambda)$ as above, in Section \ref{section:6.2}. For each $p\in J$, choose planes $\pi_1(p)\in\Pi_1, \ldots, \pi_d(p)\in \Pi_d$ so that $\delta(p,\bm{\pi}(p))=1$. If $f\in \F_\lambda[x_1, \ldots, x_n]$ is non-zero then there exists $p\in J$ and $D_1\in B(p,\pi_1(p)), \ldots, D_d\in B(p,\pi_d(p))$ so that 
\[D_1\cdots D_d f(p) \neq 0.\]
\end{lemma}
\begin{corollary}\label{corollary:4}
Fix $\alpha$ and $\lambda$. For each $1\leq j \leq d$, $\pi_j\in \Pi_j$ and $p\in J$, build the sets $B(p,\pi_j, \alpha, \lambda)$ as above, in Section \ref{section:6.2}. For each $p\in J$, choose planes $\pi_1(p) \in \Pi_1, \ldots, \pi_d(p) \in \Pi_d$ so that $\delta(p,\bm{\pi}(p)) =1$. Let $\tilde S_{k_j}(p,\pi_j,\alpha, \lambda) := \abs{B(p,\pi_j, \alpha, \lambda)}$. Then
\[\sum_{p\in J} \prod_{j=1}^d\tilde S_{k_j}(p,\pi_j(p),\alpha,\lambda) \geq {\lambda + n\choose n}.\]
\end{corollary}
\begin{proof}
For a contradiction, suppose that the conclusion is false. By parameter counting, there exists a non-zero $f\in \F_\lambda[x_1, \ldots, x_n]$ so that 
\[D_1 \cdots D_d f(p)=0\]
for all $D_j\in B(p,\pi_j(p),\alpha,\lambda)$, and all $1\leq j \leq d$, contrary to Lemma \ref{lemma:13}.
\end{proof}
This application of the vanishing lemma concludes the use of polynomials in our argument and it remains to prove that there exists a ``good'' handicap.

\subsection{Handicap Properties}\label{section:handicaps}
In this subsection we recall and remark upon the uniform boundedness, monotonicity and continuity of the numbers $\tilde S(p,l,\alpha,\lambda)$ with respect to $\alpha$, as stated below. We use this case to highlight the relation between the handicap argument and our discrete Bourgain--Guth theorem.

In the case of lines $l$, it is possible to deduce the aforementioned properties of $\tilde S(p,l,\alpha,\lambda)$ from the definition of $\prec$, alone, without reference to polynomials, or vector spaces thereof. 

The advantages of using the new Tidor--Yu--Zhao approach are twofold. Firstly, the choice of vanishing conditions varies nicely with respect to the handicap. Secondly, interpreting vanishing conditions as elements of a dual space allows us to establish the uniform equality \eqref{eq:UpperBound}, which counts the number of points ``with multiplicity'' on a plane.

The good properties of handicaps are as follows, and are proved in \cite{TYZ20}.

\begin{lemma}[$k$-Plane Uniform Boundedness]\label{lemma:kUniformBoundedness}
Let $\lambda \in \N$ and $\alpha : J \rightarrow \Z$ be a handicap. If $\alpha$ is such that $\alpha_p < \alpha_q - \lambda$ for some $p,q\in \pi \cap J$, then $\tilde S(p,\pi, \alpha,\lambda)=0$.
\end{lemma}
\begin{lemma}[$k$-Plane Monotonicity]\label{lemma:kMonotonicity}
Let $\lambda \in \N$, and $\alpha^{(1)}, \alpha^{(2)}\in \Z^J$ be two handicaps. Suppose $\exists p\in\pi \cap J$ so that $\alpha_p^{(1)} - \alpha_{p^\prime}^{(1)} \leq \alpha_p^{(2)} - \alpha_{p^\prime}^{(2)}$ for all $p^\prime\in \pi \cap J $. Then 
\[\tilde S_k(p,\pi,\alpha^{(1)}, \lambda) \leq \tilde S_k(p, \pi, \alpha^{(2)}, \lambda).\] 
\end{lemma}
\begin{lemma}[$k$-Plane Continuity]\label{lemma:kContinuity}
Let $p\in \pi \cap J$, let $\alpha^{(1)}, \alpha^{(2)}\in \Z^J$ be handicaps and $\lambda\in \N$. Then
\[\abs{\tilde S(p,\pi,\alpha^{(1)}, \lambda) - \tilde S(p,\pi, \alpha^{(2)}, \lambda)}\leq {\lambda + k-1 \choose k-1}\sum_{p^\prime \in J}\abs{(\alpha_p^{(1)} - \alpha^{(1)}_{p^\prime}) - (\alpha^{(2)}_p-\alpha^{(2)}_{p^\prime})}.\]
\end{lemma}

By translation invariance, Proposition \ref{proposition:3}, we may choose $\alpha_p^{(i)}=0$ for $i=1,2$. The resulting inequality reads as
\[\abs{\tilde S(p,\pi,\alpha^{(1)}, \lambda) - \tilde S(p,\pi, \alpha^{(2)}, \lambda)}\leq {\lambda + k-1 \choose k-1}\sum_{p^\prime \in J}\abs{\alpha^{(1)}_{p^\prime}-\alpha^{(2)}_{p^\prime}}.\]

This is more easily recognisable as a Lipschitz continuity result. In fact, if desired, we could redefine handicaps and identify $\alpha$ by its equivalence class, $\alpha + \Z^{\abs{J}}$, under equivalence by translation invariance.

\section{The Discrete Bourgain--Guth Theorem}\label{section:BG}

We now prove our discrete Bourgain--Guth theorem by making a material modification the Tidor--Yu--Zhao perturbation argument, \cite{TYZ20}.

\subsection{There Exists a Good Handicap}\label{section:GoodHandicap}We prove that there exists a handicap with  properties which are good for our purposes, and begin by introducing the notion of ``connectedness'' of the multijoint configuration, which we define as follows. We say $\pi \in \Pi_1\cups\Pi_d$ \textbf{contributes} to a multijoint $p$ if there is some tuple $(\pi_j)_j\in \Pi_1\timess\Pi_d$ so that $\pi = \pi_j$ for some $j$ and $\delta(p,(\pi_j)_j)=1$. We say that $p,q\in J$ are \textbf{adjacent} if there is some $\pi \in \Pi_1\cups \Pi_d$ that contributes to $p$ and contributes to $q$. We say a set $E\subseteq J$ is \textbf{connected} if given any $p,q\in E$, there is a sequence of points $p=p^{(1)}, \ldots, p^{(N)} = q\in E$ so that $p^{(i)}$ and $p^{(i+1)}$ are adjacent for all $1\leq i < N$. This defines an equivalence relation on any $E\subseteq J$. In order to prove Theorem \ref{theorem:2b}, it suffices assume that $\Supp S$ is connected.

Our main lemma can now be stated as follows.

\begin{lemma}[$k$-Plane Handicap]\label{lemma:kHandicap}
Let $\lambda \in \N$, and $\Pi_1, \ldots, \Pi_d$ be finite sets of $k_1$-$, \ldots, k_d$ -planes in $\F^n$ where $k_1+\ldots + k_d=n$, with associated multijoints $J$. Let $S : J\rightarrow \R_{\geq0}$ be finitely supported and suppose that any two multijoints in $\Supp S$ are connected by multijoints in $\Supp S$. Then there is a handicap $\alpha : J \rightarrow \Z$ so that for all $p\in \Supp S$, 
\begin{equation}\label{eq:24}\min_{\bm{\pi} : \delta(p,\bm{\pi})=1}\frac{1}{S(p)^d}\(\prod_{j=1}^d \frac{\tilde S_{k_j}(p,\pi_j,\alpha, \lambda)}{{\lambda+k_j\choose k_j}}\) \end{equation}
lies in a common interval with length $\leq h^\prime/\lambda$ for some $h^\prime=h^\prime(S,J,n)$, which does not depend on $\lambda$. Furthermore, we may choose $\alpha$ so that $\tilde S_{k_j}(p,\cdot,\alpha,\lambda)=0$ for all $p\not\in \Supp S$.
\end{lemma}
\begin{remark}
This lemma is different from its precursor, \cite[Lemma 5.10]{TYZ20}. Specifically, the number of factors in the geometric mean \eqref{eq:24} is now precisely $d$ rather than 
\[\sum_{\pi_1, \ldots, \pi_d} \delta(p,\pi_1, \ldots, \pi_d)f_1(\pi_1) \cdots f_d(\pi_d)\]
 terms, \textit{i.e.} the multijoint multiplicity, at each $p\in J$ in \cite{TYZ20}. Moreover, the weights for our geometric mean \eqref{eq:24} are uniform and not problem-dependent as they were in \cite{TYZ20}. \demo
\end{remark}

We now prepare for our proof of Lemma \ref{lemma:kHandicap}. Let $\alpha$ be a handicap. For any $p\in J$ and $(\pi_j)_j\in\Pi_1\timess \Pi_d$ so that $p\in \cap_j \pi_j$, let
\[W((\pi_j)_j, \alpha) := \frac{1}{S(p)^d}\prod_{j=1}^d \tilde S_{k_j}(p,\pi_j,\alpha,\lambda).\]
Note that $W$ depends additionally on $p$ and $\lambda$. However, $\lambda$ is fixed for this lemma, and the dependence on $p$ is implicit because $p\in \cap_j \pi_j$. Therefore, we suppress $p$ and $\lambda$. Now, for each $p\in J$, we define
\[w_p(\alpha) := \min_{(\pi_j)_j : \delta(p, (\pi_j)_j)=1} W((\pi_j)_j, \alpha).\]

Let $w : \{\alpha:J\rightarrow \Z\}\rightarrow \R_{\geq0}^{\abs{J}}$ be the map such that $w(\alpha) = (w_p(\alpha))_p$. Define $A=A(\lambda)\subset \{\alpha : J \rightarrow \Z\}$ to be the set of $\alpha$ such that $\tilde S(p,\cdot,\alpha,\lambda)=0$ for all $p\not\in \Supp S$. By uniform boundedness, Lemma \ref{lemma:kUniformBoundedness}, the image $w(A)$ is finite and non-empty.

Let us label each $p\in J$ so that $J = \{p_1, \ldots, p_{\abs{J}}\}$. For any $\alpha$, there exists a permutation $\sigma = \sigma_\alpha \in S_{\abs{J}}$ so that $w_{p_{\sigma(1)}}(\alpha) \geq \cdots \geq w_{p_{\sigma(\abs{J})}}(\alpha)$. Since the set $w(A)$ is finite, of all $w(\alpha)\in w(A)$, we can choose one so that $(w_{p_{\sigma(i)}}(\alpha))_{1\leq i \leq \abs{J}} \in \R_{\geq0}^{\abs{J}}$ is minimal with respect to lexicographical order on $\R^{\abs{J}}$.

Let 
\begin{equation}\label{10eq:9}
w(\alpha) = (w_{p_{\sigma(i)}}(\alpha))_{1\leq i \leq \abs{J}}\end{equation}
 be such a minimum and let $\alpha$ be a minimiser. By relabelling the indices of each $p\in J$, we may assume that $\sigma$ is the identity permutation so that
\[w_{p_1}(\alpha) \geq \cdots \geq w_{p_{\abs{J}}}(\alpha)\geq0.\]
For ease of notation, let $w_i :=w_{p_i}(\alpha)$ for each $1\leq i \leq \abs{J}$.

\begin{proposition}[Continuity of Perturbations]\label{proposition:kh}
 Let $1\leq t \leq \abs{J}$. Let
\[v = \sum_{\substack{1\leq i \leq t :\\ p_i \in \Supp S}}e_i + \sum_{\substack{i :\\ p_i\not\in \Supp S}}e_i\]
and $\alpha^\prime = \alpha - v$. There is a a constant $h$ which depends on $S, n$ and $\abs{J}$, but not  on $\lambda$, so that
\[
\abs{w_i(\alpha)-w_i(\alpha^\prime)} \leq  \frac{h}{2\lambda}\]
for all $1\leq i \leq N$, where $N = \abs{\Supp S}$.
\end{proposition}
\begin{proof}
We construct $h$ directly. Fix $1\leq i \leq N$, fix a $d$-tuple of planes $\bm{\pi}$  which realises $w_i(\alpha)$, and fix $(\pi^\prime_j)_j$ which realises $w_i(\alpha^\prime)$. Consider
\[\abs{w_i(\alpha) - w_i(\alpha^\prime)} = \abs{W((\pi_j)_j,\alpha) - W((\pi^\prime_j)_j, \alpha^\prime)}.\]
If $(\pi_j)_j = (\pi^\prime_j)_j$ then $\abs{w_i(\alpha) - w_i(\alpha^\prime)} = \abs{W((\pi_j)_j, \alpha) - W((\pi_j)_j,\alpha^\prime)}$. Otherwise, $(\pi_j)_j \neq (\pi^\prime_j)_j$, in which case
\[\abs{w_i(\alpha) - w_i(\alpha)} \leq \max_{{(\tilde\pi_j)_j} } \abs{W({ (\tilde\pi_j)_j}, \alpha) - W({(\tilde\pi_j)_j}, \alpha^\prime)},\]
where the maximum is over all ${(\tilde\pi_j)_j} \in (\argmin W(\cdot, \alpha) \cup \argmin W(\cdot, \alpha^\prime))$. Hence, we may assume
\[\abs{w_i(\alpha) - w_i(\alpha^\prime)} \leq \abs{W({(\tilde\pi_j)_j},\alpha) - W({ (\tilde\pi_j)_j}, \alpha^\prime)},\]
for some ${ (\tilde\pi_j)_j}$, which minimises either $W(\cdot, \alpha)$ or $W(\cdot, \alpha^\prime)$. However, for any $(\pi_j)_j\in \Pi_1\timess\Pi_d$ which forms a multijoint at $p=p_i$, 
\begin{equation}\label{eq:b7}
\begin{aligned}
&S(p)^d\(\prod_{j=1}^d {\lambda+ k_j\choose k_j}\)\abs{W((\pi_j)_j,\alpha)-W((\pi_j)_j, \alpha^\prime)} \\
&= \abs{\prod_{j=1}^d \(\tilde S(p,\pi_j,\alpha^\prime, \lambda) - (\tilde S(p,\pi_j,\alpha^\prime, \lambda) - \tilde S(p,\pi_j,\alpha, \lambda))\) - \prod_{j=1}^d \tilde S(p,\pi_j,\alpha^\prime, \lambda)}.\end{aligned}
\end{equation}

Taking \eqref{eq:b7}, we expand the first product so that we can cancel both occurences of $\prod_j \tilde S(p,\pi_j, \alpha^\prime, \lambda)$. We are then left with terms of the following form:
\begin{equation}\label{eq:c8}
\(\prod_{j\in A} \tilde S(p,\pi_j,\alpha^\prime,\lambda)\) \(\prod_{j^\prime\in B}  (\tilde S(p,\pi_j,\alpha^\prime, \lambda) - \tilde S(p,\pi_j,\alpha, \lambda))\),
\end{equation}
where $A\sqcup B = \{1, \ldots, d\}$ and $B\neq \emptyset$. To establish an upper bound on \eqref{eq:b7}, by the triangle inequality, it suffices to bound each such term of the form \eqref{eq:c8} separately. By continuity (Lemma \ref{lemma:kContinuity}), for any $\pi_j \in \Pi_j$ and $p\in \pi_j\cap J$,
\[\abs{\tilde S(p, \pi_j,\alpha,  \lambda) -\tilde S(p, \pi_j,\alpha^\prime, \lambda)} \leq {\lambda + k_j-1\choose k_j-1}\abs{J},\]
since $\norm{\alpha - \alpha^\prime}_{L^1(J)} =\norm{v}_{L^1(J)}\leq \abs{J}$. Combining this with the fact that each $\tilde S(p,\pi, \alpha^\prime, \lambda) = {\lambda + k_j\choose k_j}$ by construction, each term \eqref{eq:c8} is bounded by 
\[\prod_{j\in A} {\lambda + k_j\choose k_j}\prod_{j\in B}{\lambda + k_j-1\choose k_j-1}\abs{J} \sim_{n, J} \lambda^{\sum_{j\in A}k_j+\sum_{j^\prime\in B} (k_{j^\prime}-1)},\]
for sets $A,B$ so that $A\sqcup B = \{1, \ldots, d\}$ and $B \neq\emptyset$. Hence, \eqref{eq:b7} is dominated by 
\[S(p)^d\(\prod_{j=1}^d {\lambda+ k_j\choose k_j}\)\abs{W((\pi_j)_j,\alpha)-W((\pi_j)_j, \alpha^\prime)}  \leq h_{n-1}\lambda^{n-1} +  \ldots + h_1 \lambda + h_0\]
for $h_{n-1},\ldots, h_0$ sufficiently large, depending on $n$ and $\abs{J}$. This in turn is bounded above by $(h/2)\lambda^{n-1}$ for sufficiently large $h$, depending only on $h_{n-1}, \ldots, h_1$ and not depending on $\lambda$. Dividing by 
\[\(\prod_{j=1}^d {\lambda+ k_j\choose k_j}\)\sim_{n, J} \lambda^n\]
and updating  $h$ to additionally depend on $S(p)^{-d}$, we deduce
\[\abs{w_i(\alpha)-w_i(\alpha^\prime)} \leq \abs{W((\pi_j)_j,\alpha)-W((\pi_j)_j,\alpha^\prime)} \leq  \frac{h}{2\lambda}\]
for some $h$ which can be expressed as a function in $\abs{J}, n$ and the quantities $\{S(p)^{-d} : p\in J\}$, but does not depend on $\lambda$. This defines $h$.
\end{proof}

The remainder of this subsection is dedicated to proving Lemma \ref{lemma:kHandicap}. Before we embark on the proof itself, we give a brief outline of the argument. 

Since $w_1(\alpha) \geq \ldots \geq w_{\abs{J}}(\alpha)\geq 0$, it suffices to show that all the differences ${w_i(\alpha)-w_{i+1}(\alpha)}$ are small. We will prove this by contradiction. To begin, we assume there is some index $1\leq t <N$ so that ${w_{t}(\alpha)-w_{t+1}(\alpha)} > h/\lambda$, with $h$ as given by Proposition \ref{proposition:kh}. We construct a perturbation of the handicap $\alpha$. Let $\alpha^\prime$ be the perturbed handicap. The following three claims are established for the perturbation:\vspace{-0.3em}
\begin{enumerate}
\item The large entries of the tuple $w(\alpha)$ remain large, and the small entries remain small.
\item The perturbed handicap $\alpha^\prime$ is such that  $w(\alpha^\prime)\in w(A)$, so $w(\alpha) \leq w(\alpha^\prime)$.
\item If $w(\alpha)\neq w(\alpha^\prime)$ then $w(\alpha^\prime) < w(\alpha)$. Hence, $w(\alpha^\prime)=w(\alpha)$.
\end{enumerate}\vspace{-0.3em}
Thereafter, we realise that the perturbation can be applied to $\alpha^\prime$, the already perturbed handicap. Moreover, by the connectedness of $\Supp S$, there is a pair $p,q\in \Supp S$ so that $w_p(\alpha)$ is large, $w_q(\alpha)$ is small and there is a plane which contributes to both $p$ and $q$. If we perturb sufficiently many times, what results is a perturbed handicap $\alpha^\prime$ so that $w(\alpha)=w(\alpha^{\prime})$, and $\alpha_p^\prime < \alpha_q^\prime - \lambda$. By uniform boundedness (Lemma \ref{lemma:kUniformBoundedness}), $\tilde S_{k_j}(p,\pi,\alpha^\prime,\lambda)=0$, and hence $w_q(\alpha) =0$. So $0 = w_p(\alpha)\geq w_q(\alpha) \geq 0$, and hence $w_q(\alpha)=0$. However, we assumed for a contradiction that $w_p(\alpha)$ was large and that $w_q(\alpha)$ was small so that $w_p(\alpha) > w_q(\alpha)$, which is a contradiction.
\begin{proof}[Proof of Lemma \ref{lemma:kHandicap}]
Let $\Supp S = \{p_1, \ldots, p_N\}$ for some $N\leq \abs{J}$ be connected, and let $h^\prime = \abs{J}h$, where $h$ is given by Proposition \ref{proposition:kh}. Since the tuple $(w_i)_i$ has at most $\abs{{J}}$ non-zero entries, it suffices to show that
\begin{equation}\label{khyp2}
w_i - w_{i+1} \leq \frac h \lambda
\end{equation}
for all $1\leq i < N$. Indeed, summing this inequality over all $1\leq i \leq N$, we deduce $\max_i w_i - \min_i w_i \leq \abs{J}h/\lambda$. That is, the interval containing all $(w_i)_{1\leq i\leq N}$ has width $\sim_{k_j, J, S}1/\lambda$, where the implied constant does not depend on $\lambda$.

We now prove that \eqref{khyp2} holds for all $1\leq i < N$. {Suppose} for a contradiction that there is some index $1\leq i< N$ so that \eqref{khyp2} does not hold.  Let $t$ be the least such index, and for this choice of $t$, define 
\[v = \sum_{\substack{1\leq i \leq t :\\ p_i \in \Supp S}}e_i + \sum_{\substack{i :\\ p_i\not\in \Supp S}}e_i\]  
and let $\alpha^\prime = \alpha - v$. 

We say that $w_i$ is \textbf{large} if $i \leq t$, and $w_i$ is \textbf{small} otherwise.

Let $\pi\in\Pi_j$ for some $1\leq j \leq d$. Recall that $\sum_{p\in \pi} \tilde S_{k_j}(p,\pi,\alpha,\lambda) = {\lambda + k_j \choose k_j}$  for any $\alpha$. It follows that, if $\tilde S_{k_j}(p,\pi,\alpha^\prime,\lambda) > \tilde S_{k_j}(p,\pi,\alpha,\lambda)$ for some $p\in \pi\cap J$, then there exists some $p^\prime \in \pi\cap J$ so that $p^\prime \neq p$ and $\tilde S_{k_j}(p^\prime, \pi, \alpha^\prime, \lambda) < \tilde S_{k_j}(p^\prime, \pi, \alpha, \lambda)$. For those $i$ such that  $p_i\not\in \Supp S$, we have that $w_i(\alpha) = 0 = w_i(\alpha^\prime)$. Moreover, by monotonicity (Lemma \ref{lemma:kMonotonicity}), for $i$ such that $p_i\in \Supp S$, we have that if $i\leq t$ then $w_i(\alpha^\prime) \leq w_i(\alpha)$  (where the handicap is decreased) and  if $i > t$ then $w_i(\alpha^\prime) \geq w_i(\alpha)$ (where the handicap is unchanged). One inequality is strict if and only if the other is too.  By Proposition \ref{proposition:kh}, the difference $\abs{w_i(\alpha) - w_i(\alpha^\prime)}$ is bounded above by $h/(2\lambda)$ for all $1\leq i \leq N$, and by the definition of $t$, 
\begin{equation}\label{10eq:5}w_t(\alpha) -w_{t+1}(\alpha)> \frac h\lambda.\end{equation}

Let $\sigma^\prime\in S_{\abs{J}}$ be a permutation such that
\begin{equation}\label{10eq:6}w_{\sigma^\prime(1)}(\alpha^\prime) \geq \cdots \geq w_{{\sigma^\prime(\abs{J})}}(\alpha^\prime),\end{equation}
where $\sigma(i)=i$ for $N<i\leq \abs{J}$. 
\begin{claim}\label{claim:kt}If $1\leq i \leq t$, then  $1\leq \sigma^\prime(i) \leq t$, and if $t < i^\prime \leq N$, then $t < \sigma^\prime(i^\prime) \leq \abs{J}$.
\end{claim}
\begin{proof}
Suppose that $p_i, p_{i^\prime} \in \Supp S$ are such that $i\le t < i^\prime$. By Proposition \ref{proposition:kh},
\begin{eqnarray*}w_i(\alpha^\prime) - w_{i^\prime}(\alpha^\prime) &=& \(w_i(\alpha^\prime) - w_i(\alpha)\)  + \(w_i(\alpha)- w_{i^\prime}(\alpha)\)+ \(w_{i^\prime}(\alpha)  - w_{i^\prime}(\alpha^\prime)\)\\
&\geq& -\frac{h}{2\lambda} + \(w_i(\alpha)- w_{i^\prime}(\alpha)\)- \frac{h}{2\lambda} .
\end{eqnarray*}
Using inequalities \eqref{10eq:5} and \eqref{10eq:6}, this is at least
\[ \(w_{t}(\alpha)- w_{t+1}(\alpha)\)-\frac{h}{\lambda} > -\frac h \lambda + \frac h \lambda =0.\]
Hence, if $1\leq i \leq t$, then $w_{i}(\alpha^\prime) \geq w_{i^\prime}(\alpha^\prime)$ for all $i^\prime >t$. Therefore $1\leq \sigma^\prime(i) \leq t$. Similarly, if $t < i \leq N$, then $w_{i^\prime}(\alpha^\prime) \geq w_{i}(\alpha^\prime)$ for all $1\leq i^\prime \leq t$. Hence, $t < \sigma^\prime(i) \leq N$, as desired.
\end{proof}
\begin{claim}\label{claim:kw(A)}With $\alpha^\prime$ as above, $w(\alpha^\prime)\in w(A)$. Hence $w(\alpha^\prime)$ was among those considered when the minimiser $w(\alpha)$ was chosen.
\end{claim}
\begin{proof}
Since $w(\alpha) \in w(A)$,  we have that $\tilde S_{k_j}(p,\cdot,\alpha,\lambda)=0$ for all $1\leq j \leq d$ and $p\not\in \Supp S$. We may assume that $\alpha$ is such that for all $q\not\in\Supp S$, and all $\pi \in \Pi_1,\cups \Pi_d$ so that $\pi\ni q$, if $p\in \pi\cap \Supp S$, then
\[\alpha_q < \alpha_p - \lambda.\]
We will show that $\tilde S_{k_j}(q,\cdot,\alpha,\lambda)=0$ for all $1\leq j \leq d$ and $q\not\in \Supp S$. Let $q\not\in \Supp S$ and $1\leq j \leq d$. If $\pi \in \Pi_1\cups\Pi_d$  is such that $\pi\ni q$, and $p\in \pi\cap \Supp S$, then
\[\alpha_q^\prime = \alpha_q-1 < (\alpha_p-\lambda) -1 = (\alpha_p -1)- \lambda \leq \alpha_p^\prime  -\lambda.\]
We may assume that every plane in $\Pi_1\cups \Pi_d$ intersects $\Supp S$, so such $p$ exists. Hence, by uniform boundedness (Lemma \ref{lemma:kUniformBoundedness}), $\tilde S_{k_j}(q,\pi,\alpha^\prime,\lambda)=0$. Hence $w(\alpha^\prime)\in w(A)$ and so $(w_p(\alpha^\prime))_p$ was among those tuples considered when $(w_p(\alpha))_p$ was chosen.
\end{proof}

\begin{claim}\label{claim:k3}
If $w(\alpha^\prime)\neq w(\alpha)$ then $w(\alpha^\prime)< w(\alpha)$. Hence $w(\alpha) = w(\alpha^\prime)$.
\end{claim}
\begin{proof}
Suppose that there is some $1\leq i \leq \abs{J}$ so that $w_i(\alpha^\prime) \neq w_i(\alpha)$.  Then there is an $i \leq t$ so that $w_i(\alpha^\prime) < w_i(\alpha)$ and $1\leq \sigma^\prime(i) \leq t$. That is, $w_i(\alpha^\prime)$ is strictly smaller than $w_i(\alpha)$  and $w_{\sigma^{\prime}(i)}(\alpha^\prime)$ is among the $t$ largest values of $(w_i(\alpha^\prime))_{1\leq i \leq \abs{J}}$ by Claim \ref{claim:kt}. Hence
\[(w_{\sigma^\prime(i)}(\alpha^\prime))_{1\leq i \leq t} <_\text{lex} (w_i(\alpha))_{1\leq i \leq t},\]
and therefore, $w(\alpha^\prime)$ is of strictly lower lexicographical order than $w(\alpha)$. Moreover, by Claim \ref{claim:kw(A)}, $w(\alpha^\prime)\in w(A)$. This {contradicts} that $w(\alpha)\in w(A)$ was chosen to be minimising. 
\end{proof}

Thus, $w_i(\alpha^\prime) = w_i(\alpha - v)= w_i(\alpha)$ for all $1\leq i \leq \abs{J}$. We have not yet contradicted our assumption that \eqref{khyp2} is false, however we have deduced that the perturbation $\alpha \mapsto (\alpha - v)$ must leave $(w_i)_i$ unchanged. To conclude, we observe that we may return to \eqref{10eq:9}, use $\alpha^\prime= \alpha-v$, and repeat the application of Claim \ref{claim:kt}, Claim \ref{claim:kw(A)} and Claim \ref{claim:k3}. We may repeat this process $c$ times to deduce that 
\[w_i(\alpha )= w_i(\alpha-v) = w_i(\alpha - 2v) = \cdots = w_i(\alpha-cv)\] for all $1\leq i \leq \abs{J}$. By connectedness of $\Supp S$, we can find a plane $\pi\in \Pi_j$ for some $1\leq j \leq d$ contributing to distinct multijoints $p_i,p_j\in \Supp S$ for some $i$ and $j$ satisfying $i\leq t < j\leq N$. Taking $c$ sufficiently large (depending on $\lambda$) forces $w_i(\alpha^\prime) = 0$ by uniform boundedness (Lemma \ref{lemma:kUniformBoundedness}) and hence $w_i(\alpha)=0$. That is, one of the large entries of the tuple $(w_i(\alpha))_i$ is zero. Since $(w_i)_i$ is decreasing in $i$ and each $w_i$ is non-negative, $w_{i^\prime}(\alpha)=0$ for all $i\leq i^\prime\leq \abs{J}$.  This {contradicts} our assumption that $w_t(\alpha) - w_{t+1}(\alpha) > h/\lambda$, and hence \eqref{khyp2} holds for all $1\leq i < N$.
\end{proof}

\subsection{Good Vanishing Conditions Yield A Factorisation}\label{section:6.5}

We now deduce Theorem \ref{theorem:2b} in the special case where the sets $\Pi_1, \ldots, \Pi_d$ are finite. 

Let $S:J\rightarrow \R_{\geq0}$ be non-negative with $\norm{S}_{d} = 1$ and let $\lambda$ be sufficiently large. By Lemma \ref{lemma:kHandicap}, we may choose a handicap $\alpha=\alpha(\lambda)$, so that the numbers $\tilde S_{k_j}(p,\pi_j,\alpha,\lambda)$, as constructed in Section \ref{section:6.2}, are such that there is an interval containing
 \[w_p(\alpha) = \min_{\delta(p,(\pi_j)_j)=1} \frac{1}{S(p)^d}\(\prod_{j=1}^d \frac{\tilde S(p,\pi_j,\alpha, \lambda)}{{\lambda+k_j\choose k_j}}\)\]
 for all $p\in \Supp S $ with length at most $h^\prime/\lambda$. Moreover, $\tilde S(p,\cdot, \alpha,\lambda)=0$ for all $p\not\in \Supp S$. That is, we can find $w\in \R$ so that
\[w-\e \leq w_p(\alpha) \leq w \]
for all ${p\in \Supp S}$ and so that $\tilde S(p,\cdot,\alpha,\lambda) =0$ if $S(p)=0$, where $\e = h/\lambda$, and $h$ is the constant given in  Lemma \ref{lemma:kHandicap}. For each $p\in J$, let $(\pi_j(p))_j$ be a choice of $k_j$-planes which minimises $\prod_{j=1}^d \tilde S(p,\pi_j(p),\alpha,\lambda)$ over all tuples $(\pi_j)_j$ so that $\delta(p,(\pi_j)_j)=1$. Corollary \ref{corollary:4} applies, so,
\begin{equation}\label{eq:35}
 w=\sum_{p\in J}S(p)^{d}w \geq \sum_{p\in J}\prod_{j=1}^d \frac{\tilde S(p, \pi_j(p),\alpha, \lambda)}{{\lambda+k_j\choose k_j}} \geq \frac{1}{\prod_{j=1}^d{\lambda+k_j\choose k_j}}{\lambda  + n \choose n} \gtrsim_{n} 1,\end{equation}
where we have used that $w_p(\alpha) = 0$ if $p\not\in \Supp S$, and  $k_1+\ldots + k_d = n$. Hence, $w\gtrsim_{n}1$.\vspace{0.3em}

If $\lambda$ is large enough, then the length of the interval $\e =h^\prime/\lambda < w$ so that the left endpoint $w-\e$ will be strictly positive. Hence, each $B(p,\pi,\alpha, \lambda)\neq \emptyset$ for all $p\in J$ and every $\pi$ that contributes to $p$. Define 
\[s_{k_j, \lambda}(p, \pi_j) := \frac{\tilde S(p,\pi_j,\alpha, \lambda)}{{\lambda+k_j\choose k_j}}.\]
For each $p\in J$, if $\pi_j\in \Pi_j$ is such that $p\not\in \pi_j$, then set $s_{k_j, \e}(p,\pi_j)=0$. Thus, 
\[1\lesssim_{n} w \leq \frac h \lambda + \frac{1}{S(p)^d} \prod_{j=1}^d s_{k_j, \lambda}(p,\pi_j(p)) .\]
for any $p\in \Supp S$. Since $\norm{S}_d =1$, $S(p) \leq 1$ for all $p\in J$. Hence
\[S^{d}(p) \lesssim_{n} \frac h \lambda  + \prod_{j=1}^d s_{k_j, \lambda}(p,\pi_j(p)).\]
Since the tuple of planes $(\pi_j(p))_j$ minimises the right-hand side over all tuples such that $\delta(p,(\pi_j)_j)=1$, 
\begin{equation}\label{eq:36}
\delta(p, (\pi_j)_j)S^{d}(p) \lesssim_{n}\frac h \lambda +  \prod_{j=1}^d s_{k_j,\lambda}(p,\pi_j)\end{equation}
for all tuples of planes $(\pi_j)_j$. By construction, for all $1\leq j \leq d$, $\pi_j\in \Pi_j$
\begin{equation}\label{eq:37}\sum_{p\in \pi_j\cap J} s_{k_j, \lambda}(p, \pi_j) = \frac{1}{{\lambda+k_j\choose k_j}}\sum_{p\in \pi_j\cap J}\tilde S_{k_j}(p,\pi_j(p),\alpha, \lambda) =1.\end{equation}
Both inequality \eqref{eq:36} and equation \eqref{eq:37} are uniform in $\e$ and each function 
\[\frac{1}{S(p)}s_{k_j, \lambda}(p,\pi_j) = \frac{1}{S(p)}\frac{\tilde S_{k_j}(p,\pi_j,\alpha, \lambda)}{{\lambda+k_j\choose k_j}}\]
can be realised as an $\R$-valued vector in $[0,1]^{\abs{J}\times\abs{\Pi_j}}$. Hence, passing to a subsequence if necessary, we may let $s_{k_j} = \lim_{\lambda \rightarrow \infty} s_{k_j, \lambda}$ for each $1\leq j \leq d$. Letting $\lambda\rightarrow \infty$ in  \eqref{eq:36} and \eqref{eq:37} concludes our proof of Theorem \ref{theorem:2b} under the assumption that the sets $\Pi_j$ are finite. \qed

\subsection{Multijoint Factorisation to Discrete Factorisation}
With Theorem \ref{theorem:2b} proved for finite sets $\Pi_j$, we deduce Theorem \ref{corollary:01}.  If the field $\F$ is finite, then we can apply Theorem \ref{theorem:2b} with each $\pi_j = \Gr(k_j, \F^n)$. Hence, we may assume that the field is infinite. 

Let $S : \F^n \rightarrow \R_{\geq0}$ be finitely supported. Consider 
\[\calO:=\left\{\(\pi_j \cap \Supp S\)_j : \delta(p,\pi_1, \ldots, \pi_d)=1, \, p\in\F^n,\, \pi_j \subset \F^n \right\},\]
where in the definition of $\calO$, we consider any tuple of $k_j$-planes (of which there may be infinitely many) and any $p\in \Supp S$. However, since $\Supp S$ is finite, so too is $\calO$. For every $(E_j)_j\in \calO$, choose a tuple of $k_j$-planes so that $\delta(p,\pi_1, \ldots, \pi_d)=1$ and $\pi_j\cap \Supp S = E_{j}$ for every $1\leq j \leq d$. We hence define the finite sets $\Pi_j$ to consist of all such planes $\pi_j$. In particular, since the field is infinite, for each $p\in \Supp S$, the tuple $(\{p\})_j$ belongs to $\calO$ and hence $\Supp S$ is a subset of the multijoints formed by the finite sets $\Pi_1, \ldots, \Pi_d$. Hence, there exist factorising functions $\tilde s_{k_j}$ that satisfy the displays described in Theorem \ref{theorem:2b}.

Recall from Remark \ref{remark:03} that if two $k_j$-planes $\pi$ and $ \pi^\prime$ are distinct and satisfy $\pi\cap \Supp S = \pi^\prime \cap \Supp S$, then $B(\cdot, \pi,\alpha, \lambda) = B(\cdot, \pi^\prime, \alpha,\lambda)$. Hence, the functions $\tilde s_{k_j}$ satisfy $\tilde s_{k_j}(\cdot, e(\pi)) = \tilde s_{k_j}(\cdot, e(\pi^\prime))$. Now, consider the pairs $(p,V_j)\in \Supp S\times\Gr(k_j, \F^n)$. By construction, there exists $\pi_j \in \Pi_j$ so that $\pi_j \cap \Supp S = (p+V_j)\cap \Supp S$ and it is well-defined to set $s_{k_j}(p, V_j) := \tilde s_{k_j}(p, e(\pi_j))$. We additionally set $s_{k_j}(p,\cdot)=0$ for any $p\not\in\Supp S$, whereby each $s_{k_j}$ is finitely supported and defined on $\F^n \times \Gr(k_j, \F^n)$. Each $s_{k_j}$ automatically satisfies
\[\sum_{p \in \pi_j} s_{k_j}(p, e(\pi_j)) =\sum_{p \in \pi_j\cap J} s_{k_j}(p, e(\pi_j)) = 1\]
for any $k_j$-plane $\pi_j\subset \F^n$, establishing the second display of Theorem \ref{corollary:01}. Turning to the first display, let $p\in \Supp S$ and let $V_j \in \Gr(k_j, \F^n)$ be such that $V_1\wedges V_d = 1$. For each $j$ there exists $\pi_j \in \Pi_j$ so that $(p+V_j)\cap \Supp S = \pi_j\cap \Supp S$ and $\delta(p,\pi_1, \ldots, \pi_d) = 1$ by construction. Hence
\begin{eqnarray*}
(V_1\wedges V_d)S(p)^d&=&\delta(p, \pi_1, \ldots, \pi_d)S(p)^d \\
 &\lesssim_{n}& \prod_{j=1}^d \tilde s_{k_j}(p, e(\pi_j))\hspace{1em} =\hspace{1em} \prod_{j=1}^d s_{k_j}(p,V_j).\end{eqnarray*}
 This concludes the proof of Theorem \ref{corollary:01}.
 
An alternative presentation of this argument in the case where each $k_j=1$ can also be found in \cite[Section 4]{CarTan}.
\subsection{Multijoints of Varieties}
Given $k_j$-dimensional varieties $\gamma_j$, respectively, we extend the definition of $\delta$ so that for any $p\in \F^n$ which is a regular point of each $\gamma_j$,
\[\delta(p,\gamma_1, \ldots, \gamma_d) := \(\prod_{j=1}^d \chi_{\gamma_j}(p)\)e(T_p\gamma_1)\wedges e(T_p(\gamma_d),\]
where $T_p\gamma_j$ denotes the tangent plane to $\gamma_j$ at $p$.
Although we will not include a proof, a suitable modification to the argument, above, establishes Theorem \ref{theorem:6}:
\begin{theorem}[Multijoint Factorisation for Varieties]\label{theorem:6} Let $\Gamma_1, \ldots, \Gamma_d$ be sets of $k_1$-$, \ldots, k_d$-dimensional varieties in $\F^n$, respectively, where $k_1 + \ldots + k_d=n$, and let $J = \{p : \exists (\gamma_j)_j \text{ so that } \delta(p,\gamma_1, \ldots, \gamma_d)=1\}$. 

For all finitely supported $S:J\rightarrow \R_{\geq0}$ with $\norm{S}_{d}=1$  there exist  functions $s_{k}: J\times \(\cup_{j : k_j = k}\Gamma_{j}\)\rightarrow \R_{\geq0}$  for each $k\in \{k_1, \ldots, k_d\}$, so that
\[\delta(p,\gamma_1, \ldots, \gamma_d)S(p)^d \lesssim_{n} \prod_{j=1}^ds_{k_j}(p,\gamma_j),\]
 for all $p\in J$ and $(\gamma_1, \ldots, \gamma_d)\in \Gamma_1\timess \Gamma_d$, and so that
\[\sum_{p\in \gamma_j\cap J}s_{k_j}(p, \gamma_j) = \deg \gamma_j\]
for all $\gamma_j\in \Gamma_j$ and all $1\leq j \leq d$. 
\end{theorem}
Indeed, the inputs to our new Lemma \ref{lemma:kHandicap}, above, were numbers $\tilde S_{k_j}$ which satisfy: equation \eqref{eq:020}; Corollary \ref{corollary:4}; Lemma \ref{lemma:kUniformBoundedness}; Lemma \ref{lemma:kMonotonicity} and Lemma \ref{lemma:kContinuity}. To deduce Theorem \ref{theorem:6}, our proof of Lemma \ref{lemma:kHandicap} remains valid, with appropriately generalised numerology of the aforementioned inputs, and these are all given in \cite{TYZ20}.

\addcontentsline{toc}{chapter}{References}
\renewcommand{\bibname}{References}
\bibliography{ThesisBib}{}
\bibliographystyle{alpha}
\end{document}